\documentclass{amsart}
\usepackage{amssymb, amsthm, amsmath}
\theoremstyle{plain}
\newtheorem{thm}{Theorem}[section]
\newtheorem{prop}[thm]{Proposition}
\newtheorem{lem}[thm]{Lemma}
\newtheorem{cor}[thm]{Corollary}

\theoremstyle{definition}

\newtheorem{dfn}[thm]{Definition}
\newtheorem{ex}[thm]{Example}

\newtheorem*{conj}{Conjecture}

\theoremstyle{remark}
\newtheorem{rem}[thm]{Remark}

\newtheorem*{ack}{Acknowledgment}
\newcommand\m{\mathfrak{m}}
\newcommand\Hess{\operatorname{Hess}}
\newcommand\Hilb{\operatorname{Hilb}}
\newcommand\Sym{\operatorname{Sym}}
\newcommand\Ann{\operatorname{Ann}}
\newcommand\ini{\operatorname{in}}
\newcommand\LCM{\operatorname{L.C.M}}
\newcommand\F{\mathcal{F}}

\newcommand\ZZ{\mathbb{Z}}
\newcommand\FF{\mathbb{F}}
\newcommand\PP{\mathbb{P}}
\newcommand\RR{\mathbb{R}}
\newcommand\B{\mathcal{B}}

\begin{document}
\title[Gorenstein algebras associated to matroids]{Sperner property and finite-dimensional Gorenstein 
algebras associated to matroids}
\author[T. Maeno]{Toshiaki Maeno}
\address{Toshiaki Maeno, Department of Electrical Engineering, 
Kyoto University, 
Kyoto 606-8501, Japan}
\email{maeno@kuee.kyoto-u.ac.jp}
\thanks{The first author is supported by Grant-in-Aid for Scientific Research.}
\author[Y. Numata]{Yasuhide Numata}
\address{Yasuhide Numata, Department of Mathematical Informatics, 
The University of Tokyo, 
Hongo 7-3-1, Bunkyo-ku, Tokyo, 113-8656, Japan}
\address{Japan Science and Technology Agency (JST) CREST}
\email{numata@stat.t.u-tokyo.ac.jp}
\thanks{The second author is supported by JST CREST}
\date{}
\begin{abstract}
We prove the Lefschetz property for a certain class of finite-dimensional 
Gorenstein algebras associated to matroids. Our result implies 
the Sperner property of the vector space lattice. More generally, it is shown that 
the modular geometric lattice has the Sperner property. 
We also discuss the Gr\"obner fan of the defining ideal of our 
Gorenstein algebra. 
\end{abstract}
\maketitle
\section*{Introduction}
The Lefschetz property for Artinian Gorenstein rings is a ring-theoretic abstraction 
of the Hard Lefschetz Theorem for compact K\"ahler manifolds. Stanley developed 
the ideas of applications of the Lefschetz property to combinatorial problems. 
For example, he showed in \cite{S1} the Sperner property of the Bruhat ordering on the 
Weyl groups based on the Hard Lefschetz Theorem for the flag varieties. 
One of the main topics of the present paper is an application of 
the Lefschetz property for a certain kind of finite-dimensional Gorenstein 
algebras to the {\em Sperner property} of the vector space lattice $V(q,n)$ consisting of 
the linear subspaces of the vector space $\FF_q^n.$ 
A finite ranked poset $P=\bigcup_{i\geq 0} P_i$ with the level sets $P_i$ 
is said to have the Sperner property if the maximal cardinality of antichains of $P$ 
is equal to $\max_i (\# P_i).$ 

For a given ranked poset $P=\bigcup_i P_i,$ let $V_i$ be the vector space spanned 
by the elements of $P_i.$ 
The Sperner property for $P$ can be shown by constructing 
a sequence $(f_0,f_1,f_2,\ldots )$ of linear maps $f_i:V_i \rightarrow V_{i+1}$ 
satisfying a certain condition. Let $A^{(i)}=(a_{uv}^{(i)})_{u\in P_i,v\in P_{i+1}}$ 
be the matrix representing $f_i,$ i.e., 
\[ f_i(u)=\sum_{v\in P_{i+1}}a_{uv}^{(i)}v , \;\;\; u\in P_i. \] 
If every matrix $A^{(i)}$ satisfies the condition 
$a_{uv}^{(i)}\not=0 \Rightarrow u<v,$ 
and is of full rank, then $P$ has the Sperner property. 
(See e.g. \cite{HMMNWW} for details.) 
 
The Sperner property of the vector space lattice $V(q,n)$ can be deduced from 
the result on the rank of its incidence matrices due to Kantor \cite{Ka}. 
We will give another proof of the Sperner property of $V(q,n)$ by the 
construction of a finite-dimensional Gorenstein algebra $A_{M(q,n)}$ associated 
to the matroid $M(q,n)$ on the finite projective space $\PP^{n-1}(\FF_q)$ and by showing that 
$A_{M(q,n)}$ has the Lefschetz property. 

Our construction can be done for general matroids. 
For a matroid $M$ and 
its bases $\B,$ we introduce a polynomial $\Phi_M := \sum_{B\in \B}x_B.$ 
The Gorenstein algebra $A_M$ will be defined to be the quotient algebra 
of the ring of the differential polynomials by the annihilator $\Ann \Phi_M$ of 
$\Phi_M.$ We will generalize the results for the matroid $M(q,n)$ 
to the case of matroids corresponding to modular geometric lattices. 

For a general polynomial $F$, though $F$ has all the informations 
on the annihilator $\Ann F$ in principle, the combinatorial structure of 
$\Ann F$ is quite delicate in general, so it is difficult 
to describe directly from $F.$ 
It is remarkable that in our case 
the Gr\"obner fan $G(\Ann \Phi_{M(q,n)})$ of the annihilator of $\Phi_{M(q,n)}$ 
is a refinement of that of the principal ideal generated by $\Phi_{M(q,n)},$ 
which is also a consequence of our main theorem. 
As discussed in \cite{BJSST}, the Gr\"obner fan of an ideal is often 
difficult to compute. 
We will see that $G(\Ann \Phi_{M(q,n)})$ can be recovered from 
the tropical hypersurfaces of certain polynomials defined by the bases of 
the linear subspaces of $\PP^{n-1}(\FF_q).$ 

\begin{ack}
The authors thank Junzo Watanabe for suggesting the idea of the proof 
of the Sperner property for the vector space lattice via the Lefschetz 
property. They are grateful to Satoshi Murai for his helpful comment 
on modular geometric lattices. 
\end{ack}

\section{Finite-dimensional Gorenstein algebras and Lefschetz property} 
In this section we summarize some fundamental results on the structure 
of finite-dimensional Gorenstein algebras and on the Lefschetz property, 
which will be used in the subsequent sections. 
\begin{dfn}
Let $A=\oplus_{d=0}^DA_d,$ $A_D\not=0,$ be a graded Artinian algebra. 
We say that $A$ has {\em the strong Lefschetz property (in the narrow sense)} if there exists an element 
$L \in A_1$ such that the multiplication map 
\[ \times L^{D-2i}:A_i \rightarrow A_{D-i} \] 
is bijective for $i=0,\ldots,[D/2].$ 
\end{dfn}

In the rest of this paper, we consider the Gorenstein algebras that is 
finite-dimensional over a field $k$ 
of characteristic zero. 
\begin{dfn} (See \cite[Chapter 5, 6.5]{Sm}.) 
A finite-dimensional graded $k$-algebra $A=\oplus_{d=0}^DA_d$ is called the {\em Poincar\'e 
duality algebra} if $\dim_k A_D=1$ and the bilinear pairing 
\[ A_d \times A_{D-d} \rightarrow A_D \cong k \] 
is non-degenerate for $d=0,\ldots,[D/2].$ 
\end{dfn}
The following is a well-known fact (see e.g. \cite{GHMS}, \cite{HMMNWW}, \cite{MW}).
\begin{prop} \label{Poincare}
A graded Artinian $k$-algebra $A$ is a Poincar\'e duality algebra if and only if 
$A$ is Gorenstein. 
\end{prop} 

\begin{cor} \label{prodGor}
The tensor product of two graded Artinian Gorenstein $k$-algebras is again Gorenstein. 
\end{cor}

Let $P=k[x_1,\ldots,x_n]$ and $Q=k[X_1,\ldots,X_n]$ be polynomial rings over $k.$ 
We may regard $P$ as a $Q$-module via the identification 
$X_i=\partial /\partial x_i,$ $i=1,\ldots,n.$ 
For a polynomial $F(x)\in P,$ denote by $\Ann F$ the ideal of $Q$ generated 
by the differential polynomials annihilating $F,$ i.e., 
\[ \Ann F:= \{ \varphi(X) \in Q\; | \; \varphi(X)F(x)=0 \} . \] 
The following is immediate from the theory of the inverse systems 
(see \cite{BH}, \cite{Ge}, \cite{GW}). 
\begin{prop} \label{Gor} 
Let $I$ be an ideal of $Q=k[X_1,\ldots,X_n]$ and 
$A=Q/I$ the quotient algebra. Denote by $\m$ the maximal ideal 
$(X_1,\ldots,X_n)$ of $Q.$ 
Then $\sqrt{I}=\m$ and the $k$-algebra $A$ is Gorenstein if and only if 
there exists a polynomial $F \in R=k[x_1,\ldots,x_n]$ 
such that $I=\Ann_Q F.$ 
\end{prop} 

\begin{ex}
The coinvariant algebra $R_W$ of the finite Coxeter group $W$ is an example 
of the finite-dimensional Gorenstein algebra with the strong Lefschetz 
property. 
The coinvariant algebra $R_W$ is defined to be a quotient of the ring of 
polynomial functions on the reflection representation $V$ of $W$ by the ideal 
generated by the fundamental $W$-invariants. 
When $W$ is crystallographic (i.e., Weyl group), the Lefschetz property 
of $R_W$ is a consequence of the Hard Lefschetz Theorem for the corresponding 
flag variety $G/B.$ Stanley \cite{S1} has shown the Sperner property of the strong 
Bruhat ordering on $W$ from the Lefschetz property of $R_W$ (except for type $H_4$). 
The Lefschetz property of $R_W$ of type $H_4$ has been confirmed in \cite{NW}. 
Since $R_W$ is Gorenstein, it has a presentation as in Proposition \ref{Gor}. 
In fact, $R_W$ is isomorphic to the algebra $\Sym V^*/ \Ann F,$ where 
$F$ is the product of the positive roots. 
\end{ex}

\begin{dfn}
Let $G$ be a polynomial in $k[x_1,\ldots,x_n].$
When a family ${\bf B}_d=\{ \alpha^{(d)}_i \}_i$ of homogeneous polynomials of degree 
$d>0$ is given, 
we call the polynomial 
\[ \det \Big( (\alpha^{(d)}_i(X) \alpha^{(d)}_j(X) G(x))_{i,j=1}^{\# {\bf B}_d} \Big) 
\in k[x_1,\ldots,x_n] \] 
the {\em $d$-th Hessian} of $G$ with respect to ${\bf B}_d,$ 
and denote it by $\Hess_{{\bf B}_d}^{(d)}G.$ 
We denote the $d$-th Hessian simply by $\Hess^{(d)}G$ if the choice of 
${\bf B}_d$ is clear. 
\end{dfn}
When $d=1$ and $\alpha^{(1)}_j(X)=X_j,$ $j=1,\ldots ,n,$ the first Hessian 
$\Hess^{(1)}G$ coincides with the usual Hessian: 
\[ \Hess^{(1)}G= \Hess\; G:=
\det \left( \frac{\partial^2 G}{\partial x_i \partial x_j}  
\right)_{ij} . \] 

Let a finite-dimensional graded Gorenstein algebra $A=\oplus_d A_d$ have 
the presentation $A=Q/\Ann_Q F.$ 
The following gives a criterion for an element $L\in A_1$ to be a Lefschetz 
element. 
\begin{prop} {\rm (\cite[Theorem 4]{W2})}
Fix an arbitrary $k$-linear basis ${\bf B}_d$ of $A_d$ for $d=1,\ldots,[D/2].$ 
An element $L=a_1X_1+\cdots +a_nX_n \in A_1$ is a strong Lefschetz element 
of $A=Q/\Ann_Q F$ 
if and only if $F(a_1,\ldots,a_n) \not= 0$ and 
\[ (\Hess_{{\bf B}_d}^{(d)}F)(a_1,\ldots,a_n) \not= 0 \]
for $d=1,\ldots,[D/2].$ 
\end{prop}  

\begin{cor} 
If one of the Hessians $\Hess_{{\bf B}_d}^{(d)}F,$ $d=1,\ldots,[D/2],$ is 
identically zero, then $A=Q/\Ann_Q F$ does not have the strong Lefschetz 
property. 
\end{cor} 

\section{Matroids} 

\begin{dfn} A pair $(E,\F)$ of a finite set $E$ and $\F \subset 2^E$ is called a 
{\em matroid} if it satisfies the following axioms $(M1),(M2),(M3).$ \\ 
$(M1)$ $\emptyset \in \F.$ \\ 
$(M2)$ If $X\in \F$ and $Y\subset X,$ then $Y\in \F.$ \\ 
$(M3)$ If $X,Y\in \F$ and $\# X > \# Y,$ then there exists an element 
$x\in X\setminus Y$ such that $Y\cup \{ x \} \in \F.$ \\ 
Here, $\F$ is called the {\em system of independent sets.} 
\end{dfn} 

\begin{dfn} Let $M=(E,\F)$ be a matroid. \\ 
(1) A maximal element $B\in \F$ is called a {\em basis} of $M.$ 
We denote by $\B=\B(M)\subset \F$ the set of bases of $M.$ \\ 
(2) For a subset $S\subset E,$ define $r(S):=\max \{ \# F \; | \; F\in \F , F\subset S \}.$ 
The map $r:2^E \rightarrow \ZZ$ is called the {\em rank function} of $M.$ \\ 
(3) For a subset $S\subset E,$ define the {\em closure} $\sigma(S)$ of $S$ by 
\[ \sigma (S) := \{ y\in E \; | \; r(S\cup \{ y \}) = r(S) \} . \]  
We define an equivalence relation $\sim$ on $2^E$ by 
\[ S\sim T \; \Leftrightarrow \; \sigma(S) = \sigma (T) . \] 
A subset $S$ of $E$ is called a {\em flat} of $M$ if $S=\sigma(S).$ 
\end{dfn} 

\begin{ex} \label{vec-sp}
The projective space $\PP:=\PP^{n-1}(\FF_q)$ over a finite field $\FF_q$ has the structure of a matroid 
by the usual linear independence. More precisely, if we define 
the system of independence set $\F$ by 
\[ \F:= \{ F \in 2^{\PP}\; | \; \textrm{$F$ is linearly independent over $\FF_q$} \} , \] 
then $(\PP,\F)$ is a matroid. We denote it by $M(q,n).$ 
In this case, the closure $\sigma (S)$ of a subset $S\in \PP$ coincides with 
the linear subspace $\langle S \rangle$ of $\PP$ spanned by $S.$ 
\end{ex} 

\begin{lem} \label{equiv}
Let $S,T\in \F.$ Then we have 
\[ S\sim T \Leftrightarrow \{U \in \F \; | \; U\cap S=\emptyset,U \cup S \in \F \} = 
\{ U \in \F \; | \; U\cap T=\emptyset,U\cup T \in \F \} . \] 
\end{lem}

\begin{proof} 
Let $S,U$ be independent sets. If $U\cap S=\emptyset$ and $S\cup U\in \F,$ then $r(S\cup \{ y\})=r(S)+1$ 
for all $y\in U,$ and we have $U\cap \sigma(S) =\emptyset.$ 
If $U\cap S=\emptyset$ and $S\cup U\not\in \F,$ then there exists an element $y\in U$ such that 
$r(S\cup \{ y\})=r(S).$ So we have $U\cap \sigma(S) \not= \emptyset.$ 
Hence $\sigma(S)$ determines the set $\{U \in \F \; | \; U\cap S=\emptyset,U \cup S \in \F \},$ and 
vice versa. 
\end{proof}

\begin{dfn} For a given matroid $M=(E,\F),$ 
the {\em matroid polytope} $P_M$ is defined by the following system of 
inequalities: 
\[ x_e \geq 0 \;\; (e\in E), \;\;\;\;\;\; \sum_{e\in A}x_e \leq r(A) \;\; (A\in 2^E). \] 
\end{dfn}

For each independent set $F\in \F,$ we define the {\em incidence vector} 
$\vec{v}_F=(v_{F,e})_{e\in E} \in \RR^E$ as follows: 
\[ v_{F,e}:=\left\{ \begin{array}{cc} 
1, & \textrm{if $e\in F,$} \\ 
0, & \textrm{otherwise.} 
\end{array} \right. \]

\begin{prop} \label{edmonds} {\rm (Edmonds \cite{Ed})} 
The matroid polytope $P_M$ 
coincides with the convex hull of $\vec{0}$ and the {\em incidence vectors} of $\F$: 
\[ P_M={\rm conv}(\{ \vec{0} \} \cup \{ \vec{v}_F \; | \; F\in \F \} ). \]  
\end{prop}

Let $\Delta_M$ be the face of $P_M$ defined by the equation $\sum_{e\in E}x_e=r(E),$ 
which is also obtained as the convex hull of the incidence vectors 
corresponding to the bases of $M.$ 

\begin{ex} \label{non-ssc}
Let $M$ be a matroid defined by the following vectors. 
\[ \begin{array}{c|c|c|c|c} 
v_1 & v_2 & v_3 & v_4 & v_5 \\ 
1 & 0 & 0 & 1 & 0 \\ 
0 & 1 & 0 & 1 & 1 \\ 
0 & 0 & 1 & 0 & 1  
\end{array} \] 
Then $\B=\{ \{1,2,3\},\{1,2,5\},\{1,3,4\},\{1,3,5\},\{1,4,5\},\{2,3,4\},\{2,4,5\}, 
\{ 3,4,5\} \}.$ The polytope $\Delta_M$ is the convex hull 
of the following points in $\RR^5$: 
\[ (1,1,1,0,0),(1,1,0,0,1),(1,0,1,1,0),(1,0,1,0,1), \] 
\[ (1,0,0,1,1),(0,1,1,1,0),(0,1,0,1,1),(0,0,1,1,1) . \] 
\end{ex} 

\section{Gorenstein algebras associated to matroids} \label{main} 
For a matroid $M=(E,\F),$ we define a polynomial $\Phi_M \in k[x_e |e\in E]$ by 
\[ \Phi_M:= \sum_{B\in \B} x_B, \] 
where $x_B:=\prod_{b\in B}x_b.$ 
Note that the Newton polytope of $\Phi_M$ coincides with $\Delta_M$ in $\RR^E.$ 
Let $Q=Q_M=k[ \partial /\partial x_e | e\in E ]$ denote the ring of differential 
polynomials. For a subset $S\subset E,$ we put $x_S:=\prod_{e\in S} x_e$ and 
$\partial^S := \prod_{e\in S}(\partial/\partial x_e).$ 
In the subsequent part of this paper, we discuss 
the structure of the Gorenstein ring $A_M:=Q / \Ann_Q \Phi_M.$  

\begin{prop} The ideal $\Ann \Phi_M$ contains 
\[ \Lambda_M:=\{x_e^2|e\in E\} \cup \{ x_S|S\not\in \F \} \cup \{ x_A-x_{A'}| 
A,A'\in \F, A\sim A' \}. \] 
\end{prop}

\begin{proof} 
Since $\Phi_M$ is square-free and does not contain the monomials of form $x_S,$ 
$S\not\in \F,$ the ideal $\Ann \Phi_M$ contains $\{x_e^2|e\in E\}$ and 
$\{ x_S|S\not\in \F \}.$ If $A,A'\in \F$ are equivalent, then 
we have $\partial^A \Phi_M = \partial^{A'}\Phi_M$ from Lemma \ref{equiv}. 
\end{proof}

We denote by $J_M \subset Q$ the ideal generated by the set $\Lambda_M.$ 
Let $M=(E,\F)$ be a matroid, and $\F_i \subset \F$ for $i=1,\ldots, r(E),$ 
the set of independent sets of cardinality $i,$ i.e., 
\[ \F_i:= \{ F\in \F \; | \; \# F = i \} . \] 
Let $\Omega:=2^E/\sim,$ $\overline{\F}_l:=\F_l/\sim$ and $m_l:=\# \overline{\F}_l.$ 
We can identify $\Omega$ with the set of the flats of $M.$ 
Under this identification, we define the subset $\Omega(l),$ $l=1,\ldots,r(E),$ of 
$\Omega$ by 
\[ \Omega(l):= \{ S\in 2^E \; | \; S=\sigma(S), r(S) = l \}. \] 
For an equivalence class $\tau \in \Omega,$ consider a polynomial $f_{\tau}$ given by 
\[ f_{\tau}:= \sum_{F\in \F \cap \tau}x_F . \] 

\begin{prop}
We have 
\[ J_M= \mathop{\bigcap}_{\tau \in \Omega} \Ann f_{\tau}. \] 
\end{prop}

\begin{proof}
It is easy to see that $\Lambda_M$ is contained in $\cap_{\tau \in \Omega}\Ann f_{\tau}.$ 
It is enough to show that a polynomial $p\in \cap_{\tau \in \Omega}\Ann f_{\tau}$ of form 
\[ p= \sum_{\tau \in \Omega}\sum_{F \in \F \cap \tau} a_F x_F, \;\; a_F \in k, \] 
is a linear combination of polynomials of $\Lambda_M.$ 
Put $p_{\tau}:= \sum_{F \in \F \cap \tau} a_F x_F$ and consider the polynomial 
\[ p':= \sum_{\tau \in \Omega,p_{\tau}\not\in \Lambda_M}p_{\tau}. \] 
Choose $\tau_0\in \Omega$ with $p_{\tau}\not=0$ of minimum rank. Then 
\[ p(\partial)f_{\tau_0}=p_{\tau_0}(\partial)f_{\tau_0}=\sum_{F\in \F \cap \tau_0}a_F=0. \] 
Let $\F \cap \tau= \{ F_1,\ldots,F_s \}.$ 
Then we have 
\[ p_{\tau}=a_{F_1}(x_{F_1}-x_{F_2})+(a_{F_1}+a_{F_2})(x_{F_2}-x_{F_3})+ 
\cdots +(a_{F_1}+\cdots + a_{F_{s-1}})(x_{F_{s-1}}-x_{F_s}) . \] 
\end{proof}

\begin{prop} 
The subset $\Lambda_M$ of $Q$ is a universal 
Gr\"obner basis of $J_M.$ 
\end{prop}

\begin{proof} The proof is based on Buchberger's criterion. 
Fix a monomial ordering $\leq$ on the polynomial ring $Q.$ For non-zero monic polynomials 
$f,g\in Q,$ the $S$-polynomial $S(f,g)$ is given as follows: 
\[ S(f,g):= -\frac{\Gamma(f,g)}{\ini_{\leq}(f)}f+\frac{\Gamma(f,g)}{\ini_{\leq}(g)}g, 
\;\;\; \Gamma(f,g):=\LCM(\ini_{\leq}(f),\ini_{\leq}(g)) . \] 
Let $\Lambda_1:=\{ x_A-x_{A'}\; | \; A,A'\in \F, A\sim A' \},$ 
$\Lambda_2:=\{x_e^2|e\in E\}$ and $\Lambda_3:=\{ x_S|S\not\in \F \}.$ 
We will show that the $S$-polynomials $S(f,g)$ 
are reduced to zero by the division algorithm with respect to 
$\Lambda_M\setminus \{f,g\}$ for cases: \\ 
(i) $f,g\in \Lambda_1,$ (ii) $f\in \Lambda_1,$ $g\in \Lambda_2,$ 
(iii) $f\in \Lambda_1,$ $g\in \Lambda_3,$ (iv) $f,g\in \Lambda_2 \cup \Lambda_3.$ 
\medskip \\ 
Case (i): Take polynomials $f:=x_A-x_{A'},$ $g:=x_B-x_{B'} \in \Lambda_1$ with 
$x_A >x_{A'}$ and $x_B > x_{B'}.$ If $A\cap B=\emptyset,$ it is easy to see that 
$S(f,g)$ is reduced to zero. Assume that $A\cap B \not= \emptyset.$ Let 
$C:= A\cap B,$ $\hat{A}=A \setminus C$ and $\hat{B}= B \setminus C.$ Then we have 
$S(f,g)=x_{A'}x_{\hat{B}}-x_{B'}x_{\hat{A}}.$ 
Note that we have 
\begin{align*} 
r(A'\cup \hat{B})=r(A\cup \hat{B})=r(\hat{A}\cup C \cup \hat{B}), \\ 
r(B'\cup \hat{A})=r(B\cup \hat{A})=r(\hat{A}\cup C \cup \hat{B}), 
\end{align*} 
so $r(A'\cup \hat{B})=r(B'\cup \hat{A}).$ \\ 
(i-1) If $A'\cap \hat{B}\not= \emptyset,$ then $x_{A'}x_{\hat{B}}\in \Lambda_2.$ 
In this case, we have 
\[ (*) \;\;\; r(\hat{A}\cup B')=r(A'\cup \hat{B})< r(A')+r(\hat{B})=\# A' + \# \hat{B} = 
\# \hat{A} + \# B' , \] 
which means that $\hat{A}\cap B'\not= \emptyset$ or $\hat{A} \cup B'\not\in \F.$ 
Hence we also have $x_{\hat{A}}x_{B'} \in \Lambda_2 \cup \Lambda_3.$ \\ 
(i-2) Assume that $A'\cap \hat{B}=\emptyset.$ If $A'\cup \hat{B} \not\in \F,$ 
then we have $x_{A'}x_{\hat{B}}\in \Lambda_3.$ Moreover, again from 
the inequality $(*),$ we see that $x_{\hat{A}}x_{B'} \in \Lambda_2 \cup \Lambda_3.$ 
If $A'\cup \hat{B} \in \F,$ we have 
\[ r(\hat{A}\cup B')=r(A'\cup \hat{B})=r(A')+r(\hat{B}) =\# A' + \# \hat{B} = 
\# \hat{A} + \# B', \] 
which means that $\hat{A}\cup B' \in \F.$ Hence we have 
$S(f,g)=x_{A'}x_{\hat{B}}-x_{B'}x_{\hat{A}} \in \Lambda_1.$ \\ 
Case (ii): Take polynomials $f:=x_A-x_{A'} \in \Lambda_1$ and $g:= x_e^2 \in \Lambda_2$ 
with $x_A >x_{A'}.$ If $e\not\in A,$ then $S(f,g)=x_e^2x_{A'}$ is reduced to zero. 
If $e\in A,$ then 
$S(f,g)=x_ex_{A'}.$ Since $r(A'\cup \{ e \})=r(A\cup \{ e \})=r(A),$ we have 
$x_ex_{A'} \in \Lambda_2 \cup \Lambda_3.$ \\ 
Case (iii): Take polynomials $f:=x_A-x_{A'} \in \Lambda_1$ and $g:= x_B \in \Lambda_3$ 
with $x_A >x_{A'}.$ If $A\cap B= \emptyset,$ then 
$S(f,g)=x_{A'}x_B$ is reduced to zero. 
If $A\cap B \not= \emptyset,$ then $S(f,g)=x_{A'}x_{B\setminus A}.$ 
The inequality 
\[ r(A'\cup (B\setminus A))=r(A\cup (B\setminus A))=r(A\cup B)< \# (A\cup B)= 
\# (A'\cup (B\setminus A)) \] 
implies that $x_{A'}x_{B\setminus A}\in \Lambda_2 \cup \Lambda_3.$ \\ 
Case (iv): This case is easy because $\Lambda_2$ and $\Lambda_3$ are 
consisting of monomials. 
\end{proof}
 
\begin{cor} \label{hilb} 
The Hilbert polynomial of $Q/J_M$ is given by 
\[ \Hilb(Q/J_M,t)= \sum_{i=0}^{r(E)} (\# \bar{\F}_i)t^i . \] 
\end{cor}

\begin{ex} 
Let $M$ be the matroid as defined in Example \ref{non-ssc}. Then 
the ideal 
$\Ann \Phi_M$ contains an additional generator other than $\Lambda_M.$ 
In fact, we have 
\[ \Ann \Phi_M = J_M+ 
(x_{13}+x_{45}-x_{15}-x_{34}) . \] 
The Hilbert series of $Q/\Ann \Phi_M$ is $(1,5,5,1)$ and that of $Q/J_M$ 
is $(1,5,6,1).$ In particular, $Q/J_M$ is not Gorenstein. By direct computation, 
we get 
\[ \Hess \Phi_M = 8(x_1+x_4)(x_3+x_5)\Phi_M. \] 
This implies that $Q/\Ann \Phi_M$ has the Lefschetz property. 
\end{ex}

\section{Vector space lattice} 
In this section we treat the matroid $M=M(q,n)$ defined in Example \ref{vec-sp}. 
We define polynomials $\Phi_M^{(i)}:=\sum_{G\in \F_i}x_G$ for $i=1,\ldots,n.$ 
Note that $\Phi_M^{(n)}=\Phi_M.$ 
\begin{lem} \label{basis} 
For $M=M(q,n)$ and $l\leq [n/2],$  
the polynomials $\partial^{F}\Phi_M^{(2l)},$ $F\in \bar{\F}_l,$ are linearly independent 
over $k.$ 
\end{lem}

\begin{proof} 
In the following, $\langle S \rangle$ stands for a linear subspace 
in $\FF_q^n$ spanned by a subset $S\subset \PP^{n-1}(\FF_q).$ 
For $B\in \F_l$ and $0\leq i\leq l,$ define 
\[ \F_l(B,i):= \{ A\in \F_l \; | \; 
\dim (\langle A \rangle \cap \langle B \rangle )=i \} . \] 
Then we have 
$\F_l(B,l)= \{ A \in \F_l \; | \; A \sim B \}$ 
and 
\[ \F_l= \mathop{\bigcup}_{i=0}^l \F_l(B,i) . \] 
For $A,B\in \F_l,$ we also define 
\begin{eqnarray*} 
\F_l^A(B,i) & := & \{ A' \in \F_l(B,i) \; | \; 
\langle A \rangle \cap \langle A' \rangle = \{ \vec{0} \} \} \\ 
 & = & \{ A' \in \F_l(B,i) \; | \; 
A\cup A' \in \F_{2l} \}. 
\end{eqnarray*} 
For $B\in \F_l,$ consider a polynomial 
$\Phi(B,i):= \sum_{A\in \F_l(B,i)}x_A$ and a differential polynomial 
$P(B,i):= \sum_{A\in \F_l(B,i)}\partial^A.$ 
We have 
\begin{eqnarray*}
P(B,i)\Phi_M^{(2l)} & = & \sum_{A\in \F_l(B,i)}\partial^A \Phi_M^{(2l)} \\ 
 & = & \sum_{A\in \F_l(B,i)} \sum_{\substack{A'\in \F_l \\ A\cup A' \in \F_{2l}}} x_{A'} \\ 
 & = & \sum_{A'\in \F_l} \sum_{\substack{A\in \F_l(B,i) \\ A\cup A' \in \F_{2l}}} x_{A'} \\ 
 & = & \sum_{j=0}^l \sum_{A'\in \F_l(B,j)} \# 
\{ A\in \F_l(B,i) | A \cup A' \in \F_{2l}\} x_{A'} \\ 
 & = & \sum_{j=0}^l \sum_{A'\in \F_l(B,j)} \# \F_l^{A'}(B,i) x_{A'}. 
\end{eqnarray*}
Here, $\# \F_l^{A'}(B,i)$ is independent of the choice of 
$A'\in \F_l(B,j)$ for $M=M(q,n).$ 
Put $a_{ij}^B:=\# \F_l^{A'}(B,i)$ for $B\in \F_l$ and 
$A'\in \F_l(B,j).$ Now we have 
\[ P(B,i)\Phi_M^{(2l)}= \sum_{j=1}^l a_{ij}^B \sum_{A'\in \F_l(B,j)}x_{A'} 
 = \sum_{j=1}^l a_{ij}^B \Phi(B,j). \] 

If $i+j>l,$ then $\dim(\langle A \rangle \cap \langle B \rangle) + 
\dim (\langle A' \rangle \cap \langle B \rangle)= i+j >  l.$ Hence, we have 
$\dim ( \langle A \rangle \cap \langle A' \rangle \cap \langle B \rangle ) >0$ 
and $\langle A \rangle \cap \langle A' \rangle \not= \{ \vec{0} \}.$ 
This means that $a_{ij}^B=\# \F_l^{A'}(B,i)=0.$ 

Assume that $i+j=l.$ For $A\in \F_l(B,j),$ take an element $A_1 \in \F_j$ 
such that $\langle A_1 \rangle = \langle A \rangle \cap \langle B \rangle.$ 
We also take an element $A_2 \in \F_{l-j}=\F_i$ such that 
$\langle A_1 \cup A_2 \rangle = \langle B \rangle,$ and 
$A_3 \in \F_{n-l}$ such that $\langle B \cup A_3 \rangle = \FF_q^n.$ 
Put $A^*:=A_2 \cup A_3.$ Since $\dim \langle A^* \rangle =n-j\geq n-l \geq l,$ 
there exists an element $A'\in \F_l$ such that 
$\langle A^* \rangle \cap \langle B \rangle \subset 
\langle A' \rangle \subset \langle A^* \rangle .$ Since 
$\langle A' \rangle \cap \langle B \rangle = \langle A^* \rangle \cap 
\langle B \rangle = \langle A_2 \rangle,$ we can see that $A' \in \F_l^{A}(B,i).$ 
Hence we have $a_{ij}^B>0$ in this case. 

We have seen that the matrix $(a_{i,l-j}^B)_{i,j=0}^l$ is upper-triangular, 
so 
\[ \det (a_{i,l-j}^B)_{ij}=\prod_{i=0}^l a_{i,l-i}^B >0. \] 
Since the matrix $(a_{i,l-j})_{ij}$ is invertible, 
$\Phi_M(B,l)$ is written as a linear combination of 
$P(B,0)\Phi_M^{(2l)},P(B,1)\Phi_M^{(2l)},\ldots,P(B,l)\Phi_M^{(2l)},$ and hence 
it is a linear combination of the polynomials $\partial^F\Phi_M^{(2l)},$ $F\in \bar{\F}_l.$ 
On the other hand, it is easy to see the linear-independency of the 
polynomials $\Phi_M(B,l),$ $B\in \bar{\F}_l.$ Therefore 
the polynomials $\partial^F \Phi_M^{(2l)},$ $F\in \bar{\F}_l,$ are linearly independent. 
\end{proof} 

\begin{thm} \label{main}
Let $M=M(q,n).$ 
Take a representative $F_1,\ldots,F_{m_l}\in \F_l$ of $\overline{\F}_l.$  
Then the determinant of the matrix 
\[ \left( \partial^{F_i}\partial^{F_j} \Phi_M \right)_{i,j=1}^{m_l} \] 
is not identically zero. 
\end{thm}

\begin{proof} 
For $F\in \F_j,$ define $c(F,i):=\# \{ F'\in \F_i\; | \; F\cup F' \in \F_{i+j} \} .$ 
Then 
the equality $c(F_1,i)=c(F_2,i)$ holds for any $F_1,F_2\in \F_j$ and 
for $j=1,\ldots,r(E)-1.$ 
It is easy to see that 
\[ \det \left( \partial^{F_i}\partial^{F_j} \Phi_M \right)_{i,j=1}^{m_l}\Big|_{x=1} = 
\gamma \cdot \det \left( \delta_{\sigma(F_i), \sigma(F_j)}\right)_{i,j}, \] 
where $\gamma =c(F,l)^{m_l}\not=0$ for any $F\in \F_l,$ and $\delta_{\tau_1,\tau_2},$ $\tau_1,\tau_2\in \Omega(l),$ 
is defined by 
\[ \delta_{\tau_1,\tau_2}:=\left\{ 
\begin{array}{cc} 
1, & \textrm{if $\tau_1 \cap \tau_2 = \emptyset,$} \\ 
0, & \textrm{otherwise.} 
\end{array}
\right. \] 
At the same time, we have 
\[ \det \left( \partial^{F_i}\partial^{F_j} \Phi_M^{(2l)} \right)_{i,j} = 
\det \left( \delta_{\sigma(F_i), \sigma(F_j)}\right)_{i,j}. \] 
Note that the algebra $B^{(2l)}:=Q/\Ann \Phi_M^{(2l)}$ is also Gorenstein, 
and the natural pairings 
\[ \langle \; , \; \rangle :B_i^{(2l)} \times B_{2l-i}^{(2l)} \rightarrow 
B^{(2l)}_{2l} \cong k \] 
are non-degenerate for $i=0,\ldots,l.$ From Lemma \ref{basis}, we see that 
$\{ x_{F_i}| i=1,\ldots,m_l \}$ gives a basis of $B_l^{(2l)}.$ 
Since the matrix 
$\left( \partial^{F_i}\partial^{F_j} \Phi_M^{(2l)} \right)_{i,j}$ represents 
the pairing $\langle \; , \; \rangle$ at the intermediate 
part $B_l^{(2l)}\times B_l^{(2l)}\rightarrow k,$ we see that its determinant is non-zero. 
Therefore, $\det \left( \partial^{F_i}\partial^{F_j} \Phi_M \right)\Big|_{x=1}$ 
is non-zero, and hence it cannot be identically zero. 
\end{proof}

\begin{cor} \label{main1}
$(1)$ The algebra $A_{M(q,n)}$ has the strong Lefschetz property. \\ 
$(2)$ The ideal $\Ann \Phi_{M(q,n)}$ is generated by $\Lambda_{M(q,n)},$ i.e., 
$\Ann \Phi_{M(q,n)}=J_{M(q,n)}.$ In particular, 
it is a binomial ideal. \\ 
$(3)$ We have 
\[ \Hilb(Q/\Ann \Phi_{M(q,n)},t)= \sum_{i=0}^n t^i {n \choose i}_q , \] 
where ${n \choose i}_q,$ $0\leq i \leq n,$ are $q$-binomial coefficients. \\ 
$(4)$ The vector space lattice $V(q,n)$ consisting of the linear subspaces of 
$\FF_q^n$ has the Sperner property. 
\end{cor}

\begin{rem}
For $i\leq n,$ let $M^{(i)}(q,n)$ be a matroid structure on $\PP^{n-1}(\FF_q)$ 
obtained by regarding $\F_i$ as a system of bases. We see that 
$\Phi_{M^{(i)}(q,n)}=\Phi_{M(q,n)}^{(i)}.$ It can be shown by a similar manner 
as the proof of Theorem \ref{main} that 
$Q/\Ann \Phi_{M^{(i)}(q,n)}$ has the Lefschetz property, and 
$\Ann \Phi_{M^{(i)}(q,n)}=J_{M^{(i)}(q,n)}.$ 
\end{rem}

\begin{ex} \label{boolean} Let $[n]:=\{1,2,\ldots,n \}$ be an $n$-element set. 
The set $2^{[n]}$ of the subsets of $[n]$ has a natural lattice structure 
induced by the operations $\cup$ and $\cap.$ The obtained lattice is called 
the Boolean lattice. 
Sperner's theory originates his work \cite{Sp} on the maximal cardinality of 
the antichains of the Boolean lattice. On the other hand, $M([n]):=([n],2^{[n]})$ 
satisfies the axioms of the matroid. 
The matroid $M([n])$ has the unique basis $[n],$ so the corresponding 
Gorenstein algebra is given by 
\[ A_{M([n])}=k[X_1,\ldots,X_n]/\Ann (x_1\cdots x_n) . \] 
In \cite{HW}, it has been proved that 
$M([n])$ is another example of matroids for which Theorem \ref{main} 
holds. As a consequence, we obtain $\Ann \Phi_{M([n])}=J_{M([n])}$ 
and the Lefschetz property for $A_{M([n])},$ which gives another 
proof of the Sperner property for the Boolean lattice. 
\end{ex}

\begin{conj} 
The algebra $A_M$ has the strong Lefschetz property for an arbitrary matroid 
$M.$ 
\end{conj} 

\section{Modular geometric lattice} 
In this section, we discuss a characterization of the matroids for which 
the algebra $Q/J_M$ is Gorenstein. 

\begin{dfn} Let $L$ be a finite graded lattice with the rank function $r.$ \\ 
$(1)$ The lattice $L$ is called {\em (upper) semimodular} if 
$r(x) + r(y) \geq r(x \wedge y) + r (x \vee y)$ 
for all $x,y\in L.$ If the equality holds for all $x,y\in L,$ then $L$ is 
called {\em modular}. \\ 
$(2)$ Assume that $L$ has the unique minimal element $\hat{0}.$ An element of $L$ is 
called an {\em atom} if it covers $\hat{0}.$ The term {\em coatom} is dually defined as 
an element covered by the unique maximal element $\hat{1}.$ 
The lattice $L$ is {\em atomic} if 
every element of $L$ is written as a join of atoms. \\ 
$(3)$ The lattice $L$ is said to be {\em geometric} if $L$ is atomic and semimodular. 
\end{dfn}

The set of the flats of a matroid forms a lattice, which we denote by $L(M).$ 
It is known that a finite lattice $L$ is geometric if and only if 
$L \cong L(M)$ for a matroid $M$ (see \cite[Theorem 3.8]{S2}). 

\begin{prop} \label{Greene} {\rm (Greene \cite{Gr})} 
Let $L$ be a finite geometric lattice. The sets of atoms and of coatoms 
have the same cardinality if and only if $L$ is modular. 
\end{prop} 

Greene's characterization of the modular geometric lattice implies the following. 
\begin{prop} \label{gor-mgl}
If $Q/J_M$ is Gorenstein, then $L(M)$ is a modular geometric lattice. 
\end{prop}
\begin{proof} 
Let $n$ be the dimension of $M.$ Then the socle degree of $Q/J_M$ is $n.$ 
Suppose that $Q/J_M$ is Gorenstein. From Proposition \ref{Poincare}, 
the part $(Q/J_M)_1$ of degree 1 is isomorphic to $(Q/J_M)_{n-1}$ of degree 
$n-1$ as vector spaces. 
Since 
\[ \# \{ \textrm{atoms of $L(M)$} \} = \dim (Q/J_M)_1 = \dim (Q/J_M)_{n-1}= 
\# \{ \textrm{coatoms of $L(M)$} \}, \] 
we can conclude that the lattice $L(M)$ 
is a modular geometric lattice by Proposition \ref{Greene}. 
\end{proof}

The fundamental theorem of projective geometry shows that 
a modular geometric lattice decomposes into a direct product of boolean lattices, 
vector space lattices and incidence lattices of 
(non-Desarguesian) finite projective planes (see e.g. \cite{S2}). 

\begin{prop} \label{finproj}
Let $M(\Pi)$ be the matroid associated to a finite projective plane $\Pi.$ 
Then we have $J_{M(\Pi)}=\Ann \Phi_{M(\Pi)}.$ 
\end{prop}
\begin{proof} 
Let $\Pi$ be a projective plane of order $\nu.$ 
Since $J_{M(\Pi)} \subset \Ann \Phi_{M(\Pi)},$ we have a surjective 
homomorphism $\varphi : Q/J_{M(\Pi)} \rightarrow A_{M(\Pi)}.$ From Corollary \ref{hilb}, 
we have $\dim(Q/J_{M(\Pi)})_1 = \dim (Q/J_{M(\Pi)})_2= \nu^2+\nu+1.$ 
Hence, in order to show that $\varphi$ is an isomorphism, it is enough to 
see $\dim(Q/J_{M(\Pi)})_1= \dim (A_{M(\Pi)})_1.$ 
For two distinct points $p,q\in \Pi,$ denote by $L_{pq}$ 
the line passing through $p$ and $q.$ 
We have 
\[ \partial^p \partial ^q \Phi_{M(\Pi)} = \sum_{r \not\in L_{pq}} x_r, \] 
for $p\not=q.$ Consider the specialization $S$ of the matrix 
$(\partial^p \partial ^q \Phi_{M(\Pi)})_{p,q \in \Pi}$ at $x_a=1$ for all 
$a\in \Pi.$ Then we have 
\[ S_{pq}= \left\{ \begin{array}{cc} 
0, & \textrm{if $p=q,$} \\ 
\nu^2 , & \textrm{if $p\not=q,$} 
\end{array} \right. \] 
and $\det S \not=0.$ So the polynomials $\partial^p \Phi_{M(\Pi)},$ 
$p\in \Pi,$ are linearly independent. This shows 
$\dim(Q/J_{M(\Pi)})_1= \dim (A_{M(\Pi)})_1.$ 
\end{proof}

\begin{cor} 
The algebra $A_{M(\Pi)}$ has the strong Lefschetz property. 
\end{cor}

The following lemma is easy. 
\begin{lem} \label{prod}
If $M$ is the direct sum of two matroids $M_1$ and $M_2,$ then 
$Q_M/J_M \cong Q_{M_1}/J_{M_1} \otimes Q_{M_2}/J_{M_2}.$ 
\end{lem}

\begin{thm}
The algebra $Q/J_M$ is Gorenstein 
if and only if $L(M)$ is a modular geometric lattice. 
\end{thm} 
\begin{proof} 
In Proposition \ref{gor-mgl}, we have proved that $L(M)$ is a modular geometric lattice 
if $Q/J_M$ is Gorenstein. 

Conversely, assume that $L(M)$ is a modular geometric lattice. 
Then $L(M)$ decomposes into a direct product of boolean lattices $2^{[n]},$ 
vector space lattices $V(q,n)=L(M(q,n))$ and incidence lattices of 
finite projective planes $\Pi.$ 
For the boolean lattice $2^{[n]},$ we have seen in Example \ref{boolean} 
that $Q/J_{M([n])}$ is Gorenstein. 
For the matroid $M(q,n),$ it has been shown 
in Corollary \ref{main1} (2) that $J_{M(q,n)}=\Ann \Phi_{M(q,n)},$ so 
$Q/J_{M(q,n)}$ is Gorenstein. 
In Proposition \ref{finproj}, we see that $Q/J_{M(\Pi)}$ is Gorenstein 
for a finite projective plane $\Pi.$ Hence, from Corollary \ref{prodGor} 
and Lemma \ref{prod}, the algebra $Q/J_M$ is Gorenstein. 
\end{proof}

\begin{cor} 
$(1)$ If $L(M)$ is a modular geometric lattice, then 
$A_M$ has the strong Lefschetz property. \\ 
$(2)$ Every modular geometric lattice has the Sperner property. 
\end{cor} 

\section{Gr\"obner fan of $J_M$} 
In this section, we discuss the Gr\"obner fan of the ideals $J_M$ and 
$\Ann \Phi_{M(q,n)}.$ 
The initial ideal $\ini_{\vec{\omega}}(I)$ of an ideal $I\subset Q$ with 
respect to the weight vector $\vec{\omega}\in \RR^E$ is given by 
\[ \ini_{\vec{\omega}}(I):=( \ini_{\vec{\omega}}(f) \; | \; f\in I , f\not=0). \] 
For a weight vector $\vec{\omega},$ the set 
$C(\vec{\omega}):={\rm closure}\{ \vec{\lambda} \in \RR^E \; | \; \ini_{\vec{\lambda}}(I) = 
\ini_{\vec{\omega}}(I) \}$ is a polyhedral cone in $\RR^E.$ 
The set of cones $\{ C(\vec{\omega})\; | \; \vec{\omega}\in \RR^E\setminus \{ \vec{0}\} \}$ 
forms a fan $G(I).$ The fan $G(I)$ is called the {\em Gr\"obner fan} of $I.$ 
Denote by $G^d(I)$ the set of $d$-dimensional cones in $G(I).$ 
The Gr\"obner fan $G(I)$ of a homogeneous ideal $I$ 
has the translation invariance in the direction of 
$\vec{n}:=(1,\ldots,1)\in \RR^E.$ 
Let $H$ be the 
hyperplane in $\RR^E$ defined by the equation $\sum_{e\in E}x_e=0.$ 
Denote by 
$\bar{G}(I)$ the restriction of $G(I)$ to $H.$ 

For two distinct independent sets $F,F'\in \F$ with $F \sim F',$ define a cone 
$W_{F,F'}$ by the condition 
\[ \sum_{e\in F}x_e=\sum_{e\in F'}x_e, \;\;\; 
\sum_{e\in F}x_e \leq \sum_{e\in F''} x_e \;\; 
\textrm{($\forall F'' \in \F,$ $F''\sim F$).} \] 
Let $C_1,\ldots,C_p$ be the closures of the connected components of 
\[ \RR^E \setminus \bigcup_{\substack{F,F'\in \F \\ F \sim F',F\not= F'}} W_{F,F'}. \] 

\begin{prop} \label{gfan} 
The maximal cones of $G(J_M)$ are given by $C_1,\ldots,C_p,$ i.e., 
$G^{\# E}(J_M)=\{ C_1,\ldots,C_p \}.$ 
\end{prop}

\begin{proof} 
Since $\Lambda_M$ is a universal Gr\"obner basis of $J_M,$ 
$\ini_{\vec{\omega}}(J_M)$ is not a monomial ideal if and only if 
$\ini_{\vec{\omega}}(J_M)$ contains $x_F-x_{F'}$ for two distinct 
independent sets $F,F'$ with $F \sim F'$ and does not contain 
$x_F$ or $x_{F'}.$ This is the case when $\vec{\omega} \in W_{F,F'}.$ 
\end{proof} 

The {\em tropical hypersurface} $V_{\rm trop}(\Phi_M)\subset \RR^{E}$ 
is defined as the locus in $\RR^E$ where the piecewise linear function 
\[ {\rm trop}(\Phi_M)=\max \left( \sum_{e\in B} x_e \; | \; B \in \B \right) \] 
is not smooth. The tropical hypersurface $V_{\rm trop}(\Phi_M)$ can be considered as 
a subcomplex of $G(\Phi_M)$ (see \cite{BJSST}). 
Since $\Phi_M$ is homogeneous, the corresponding tropical hypersurface 
$V_{\rm trop}(\Phi_M)$ has the translation invariance in the direction 
of the vector $\vec{n}.$ 
Denote by $\bar{V}_{\rm trop}(\Phi_M)$ 
the restriction of $V_{\rm trop}(\Phi_M)$ to $H.$ 
In our case, $\bar{V}_{\rm trop}(\Phi_M)$ is also regarded as a fan. 
The following proposition shows that the tropical variety $\bar{V}_{\rm trop}(\Phi_M)$ 
is directly obtained from the matroid polytope of $M.$ 
\begin{prop} 
The piecewise linear function ${\rm trop}(\Phi_M)|_H$ is a support function for 
the polytope $\Delta^0_M:=\Delta_M - r(E)(\# E)^{-1} \cdot \vec{n} \subset H.$ 
\end{prop} 

\begin{proof} 
The polytope $\Delta^0_M$ is spanned by the vectors 
$\vec{u}_B:=\vec{v}_B-r(E)(\# E)^{-1} \cdot \vec{n},$ $B\in \B,$ 
by Proposition \ref{edmonds}. 
We also have the inequality 
\[ \langle \vec{u}_B,\vec{y} \rangle=\sum_{b\in B}y_b 
\leq {\rm trop}(\Phi_M)(\vec{y}) , \;\; \forall \vec{y}=(y_e)_{e\in E} \in H, \] 
and for $\vec{y}=\vec{u}_B,$ 
\[ \langle \vec{u}_B,\vec{u}_B \rangle= r(E)-\frac{r(E)^2}{\# E}= 
{\rm trop}(\Phi_M)(\vec{u}_B) . \] 
Hence, the polytope $\Delta_M^0$ is described as 
\[ \Delta_M^0= \{ \vec{x}\in H \; | \; \langle \vec{x},\vec{y} \rangle \leq 
{\rm trop}(\Phi_M)(\vec{y}), \; \forall \vec{y} \in H \}. \] 
\end{proof}

For a fan $\Sigma,$ define $-\Sigma:=\{ -\sigma |\sigma \in \Sigma \}.$ 
\begin{prop} \label{hypersurf} 
$(1)$ For an equivalence class $\tau \in \Omega(l)$ with $l\geq 2,$ we have 
\[ G^{\# E-1}(f_{\tau})= \{ -W_{F,F'} | F,F'\in \F \cap \tau, F\not= F' \}. \] 
$(2)$ \[ V_{\rm trop}(f_{\tau})=\bigcup_{\sigma \in G^{\# E-1}(f_{\tau})}\! \sigma 
 = \bigcup_{\substack{F,F'\in \F \cap \tau \\ F \not= F'}} -W_{F,F'}. \] 
$(3)$ \[ \bigcup_{\sigma \in G^{\# E-1}(J_M)}\!\! -\sigma = 
\bigcup_{\tau \in \Omega} V_{\rm trop}(f_{\tau}). \] 
\end{prop}

\begin{proof} 
Since the Newton polytope of $f_{\tau}$ does not contain 
interior lattice points, every monomial $x_F,$ $F\in \F \cap \tau,$ appearing in $f_{\tau}$ 
can be the initial monomial for a choice of monomial ordering. 
Hence, $\ini_{\vec{\omega}}(f_{\tau})$ is not a monomial ideal if 
$\vec{\omega}$ belongs to $-W_{F,F'}$ for a pair 
$F,F'\in \F \cap \tau,$ $F\not= F'.$ This shows (1). 
The second claim (2) follows from the definition of the tropical 
hypersurface $V_{\rm trop}(f_{\tau}).$ The claim (3) is a consequences of 
(2) and Proposition \ref{gfan}. 
\end{proof}

\begin{cor}
The tropical hypersurface $V_{\rm trop}(\Phi_M)$ is a subcomplex of the fan $-G(J_M).$ 
\end{cor}

For $M=M(q,n),$ we have $G(\Ann \Phi_{M(q,n)})=G(J_{M(q,n)})$ from Corollary 
\ref{main1} (2). By Proposition \ref{hypersurf}, 
the Gr\"obner fan $G(\Ann \Phi_{M(q,n)})$ can be computed from the tropical 
hypersurfaces $V_{\rm trop}(f_{\tau}).$ 

\begin{ex}
The matroid $M(2,2)$ is defined by the following 3 vectors, 
\[ \begin{array}{c|c|c}
v_1 & v_2 & v_3 \\ 
1 & 0 & 1 \\ 
0 & 1 & 1  
\end{array} \] 
so we have 
\begin{align*} 
 & \Phi_{M(2,2)}=x_1x_2+x_1x_3+x_2x_3,& \\ 
& \Ann \Phi_{M(2,2)}=(x_1^2,x_2^2,x_3^2,x_1x_2-x_1x_3,x_1x_2-x_2x_3,x_1x_3-x_2x_3). & 
\end{align*} 
In this case, the Gr\"obner fans $G(\Ann \Phi_{M(2,2)}),$ $G(J_{M(2,2)})$ and 
$-G(\Phi_{M(2,2)})$ are same. Their restrictions 
$\bar{G}(\Ann \Phi_{M(2,2)}),$ $\bar{G}(J_{M(2,2)}),$ 
$-\bar{G}(\Phi_{M(2,2)})$ to the plane $H$ are determined by 
3 rays: 
\[ R_1:= \RR_{\geq 0} (-2, 1, 1), \; 
R_2:=\RR_{\geq 0}(1, -2, 1), \; 
R_3:=\RR_{\geq 0}(1, 1, -2). \] 
Moreover, $\bar{V}_{\rm trop}(\Phi_{M(2,2)})=(-R_1)\cup (-R_2)\cup (-R_3).$ 
\end{ex}

\begin{ex}
The Gr\"obner fan $\bar{G}(\Ann \Phi_{M(2,3)})=\bar{G}(J_{M(2,3)})$ contains 
420 cones of maximal dimension 6 and 49 rays. The fan $\bar{G}(\Phi_{M(2,3)})$ contains 
28 maximal cones and 21 rays. 
\end{ex}

\begin{ex} Let $M$ be the matroid from Example \ref{non-ssc}. 
The fan $\bar{G}(J_M)$ contains 12 cones of maximal dimension 4 and 7 rays: 
\begin{align*} 
& \RR_{\geq 0} (-4, 1, 1, 1, 1),
\RR_{\geq 0}(-2, -2, 3, -2, 3),
\RR_{\geq 0}(-1, 4, -1, -1, -1),
\RR_{\geq 0}(1, 1, -4, 1, 1), & \\ 
& \RR_{\geq 0}(1, 1, 1, -4, 1), \; 
\RR_{\geq 0}(1, 1, 1, 1, -4), \; 
\RR_{\geq 0}(3, -2, -2, 3, -2). & 
\end{align*} 
The fan $\bar{G}(\Phi_M)$ contains 8 maximal cones, and 
$\bar{G}^1(\Phi_M)=-\bar{G}^1(J_M).$ 
In this case, $\bar{G}(\Ann \Phi_M)$ is a refinement of $\bar{G}(J_M).$ 
The fan $\bar{G}(\Ann \Phi_M)$ contains 20 maximal cones and 9 rays: 
\begin{align*} 
&\RR_{\geq 0}(-4, 1, 1, 1, 1), \; \RR_{\geq 0}(-3, 2, 2, -3, 2), \; \RR_{\geq 0}(-2, -2, 3, -2, 3), & \\ 
&\RR_{\geq 0}(-1, 4, -1,-1, -1), \; \RR_{\geq 0}(1, 1, -4, 1, 1), \; \RR_{\geq 0}(1, 1, 1, -4, 1), & \\  
&\RR_{\geq 0}(1, 1, 1, 1, -4), \; \RR_{\geq 0}(2, 2,-3, 2, -3), \; \RR_{\geq 0}(3, -2, -2, 3, -2). &
\end{align*}  
\end{ex}

\end{document}